\documentclass{amsart}

\usepackage{empheq}
\usepackage{color}
\usepackage[colorlinks,linkcolor=blue,anchorcolor=blue,citecolor=blue]{hyperref}
\usepackage{amssymb,amstext, amsbsy, amscd}
\usepackage[mathscr]{eucal}
\usepackage{times}
\usepackage{amsmath, amsthm, amsfonts,extarrows}

\usepackage{enumitem}

\newtheorem{thm}{Theorem}[section]
\newtheorem{prop}[thm]{Proposition}
\newtheorem{lemma}[thm]{Lemma}

\newtheorem{remark}[thm]{Remark}

\newtheorem{prob}[thm]{Problem}

\newtheorem{prop-defn}[thm]{Proposition/Definition}

\usepackage{comment}

\newlist{casesp}{enumerate}{3} 
\setlist[casesp]{align=left, 
                 listparindent=\parindent, 
                 parsep=\parskip, 
                 font=\normalfont\bfseries, 
                 leftmargin=0pt, 
                 labelwidth=0pt, 
                 itemindent=.4em,labelsep=.4em, 
                 partopsep=0pt, 
                 }
\setlist[casesp,1]{label=Case~\Roman*:,ref=\Roman*}
\setlist[casesp,2]{label=Case~\thecasespi.\arabic*:,ref=\thecasespi.\arabic*}
\setlist[casesp,3]{label=Case~\thecasespii.\alph*:,ref=\thecasespii.\alph*}

\newcommand{\bC}{{\mathbb C}}

\newtheorem*{LBconj}{Galkin's Lower Bound Conjecture}

\begin{document}{\allowdisplaybreaks[4]

\title{On Galkin's Lower Bound Conjecture}

\date{
      }



\author{Jianxun Hu }
\address{School of Mathematics, Sun Yat-sen University, Guangzhou 510275, P.R. China}

\email{stsjxhu@mail.sysu.edu.cn}
\thanks{ 
 }

\author{Huazhong Ke}
\address{School of Mathematics, Sun Yat-sen University, Guangzhou 510275, P.R. China}
\email{kehuazh@mail.sysu.edu.cn}
\thanks{
 }

\author{Changzheng Li}
 \address{School of Mathematics, Sun Yat-sen University, Guangzhou 510275, P.R. China}
\email{lichangzh@mail.sysu.edu.cn}

\author{Zhitong Su}
\address{School of Mathematics, Sun Yat-sen University, Guangzhou 510275, P.R. China}
\email{suzht@mail.sysu.edu.cn}
\thanks{
 }


\begin{abstract}
      We estimate an upper bound of the spectral radius of a linear operator on the quantum cohomology of the toric Fano manifolds     $\mathbb{P}_{\mathbb{P}^{n}}(\mathcal{O}\oplus\mathcal{O}(3))$.
       This provides a negative answer to Galkin's lower bound conjecture.
  \end{abstract}

\maketitle

\section{Introduction}

The first Chern class of a Fano manifold $X$ induces a linear operator $\hat c_1$ on  the even part $H^{\rm ev}(X)$ of the classical cohomology ring $H^*(X,\mathbb{C})$ by
  $$\hat c_1: H^{\rm ev}(X)\longrightarrow H^{\rm ev}(X); \alpha\mapsto (c_1(X)\star \alpha)|_{\mathbf{q}=1}.$$
  Here $\star$ denotes the quantum multiplication, which involves   genus-zero, three-pointed Gromov-Witten invariants of $X$, and $\mathbf{q}$ denote the quantum variables.
It is  important  to study the  distribution of eigenvalues of $\hat c_1$. Indeed in \cite{GGI}, Galkin, Golyshev and Iritani  proposed  remarkable Conjecture $\mathcal{O}$ and the relevant Gamma conjecture I and II. Conjecture $\mathcal{O}$
  concerns about the spectral radius
  $$\rho=\rho(\hat c_1):=\max\{|\lambda|\mid \lambda \mbox{ is an eigenvalue of } \hat c_1\}.$$
There is another relevant lower bound  conjecture proposed by Galkin \cite{Gal}.
 \begin{LBconj}
      For any Fano manifold $X$, we have
      $$\rho\geq \dim X+1 \mbox{ with equality if and only if } X\cong \mathbb{P}^n.$$
 \end{LBconj}

\noindent This conjecture   has been verified for some cases \cite{ESSSW, ChHa, ShWa, Ke, HKLS}, and can also be supported by the numerical computations based on the analysis in \cite{Yang} for the blowup of $\mathbb{P}^n$ along $\mathbb{P}^r$.

   As the main result of this paper, we provide a negative answer to the above conjecture, by showing the following
   for the toric Fano manifolds $\mathbb{P}_{\mathbb{P}^{n}}(\mathcal{O}\oplus\mathcal{O}(3)) $.
 \begin{thm}\label{mainthm} Let  $X= \mathbb{P}_{\mathbb{P}^{n}}(\mathcal{O}\oplus\mathcal{O}(3))$, where   $n$ is sufficiently large with $3\nmid n+1$.
Then    $\rho <\dim X+1=n+2$.
\end{thm}
\noindent We will give a more precise statement in \textbf{Theorem \ref{mainthm2}}.  To prove the theorem, we use mirror symmetry for toric Fano manifolds, especially the fact  \cite{Auroux, OsTy} that the eigenvalues of $\hat c_1$ are in one-to-one correspondence with the critical values of  the Hori-Vafa superpotential \cite{HoVa}   mirror to the toric Fano manifold $X$. We further reduce the spectral radius to an optimization problem in nonlinear programming, and then achieve the aim by classical analysis. This approach was used in  \cite{GHIKLS}, leading to negative answers to Conjecture $\mathcal{O}$ and Gamma conjecture I therein.

 \subsection*{Acknowledgements}

The authors would like to thank Sergey Galkin and Hiroshi Iritani   for  helpful discussions. Z. Su would like to thank   Xiaowei Wang for constant encouragement.
 The authors are  supported   in part by the National Key Research and Development Program of China No. 2023YFA100980001.
H. Ke is also supported in part by   NSFC Grant 12271532.



\section{Galkin's lower bound conjecture for $\mathbb{P}_{\mathbb{P}^n}(\mathcal{O}\oplus \mathcal{O}(3))$}

  This section is devoted to finding an upper bound of   the spectral radius for the   toric Fano manifolds
 $X=\mathbb{P}_{\mathbb{P}^n}(\mathcal{O}\oplus\mathcal{O}(3))$ with  $n$ sufficiently large,  as we will see in Theorem \ref{mainthm2}.


\subsection{Mirror symmetry for $X$}
Note that $X$ is a $\mathbb{P}^1$-bundle over $\mathbb{P}^n$, which is of dimension $n+1$ and of Picard number two. The classical cohomology ring $H^*(X)=H^*(X, \mathbb{C})$ is of dimension $2n+2$, and   $H^*(X)=H^{\rm ev}(X)$ does not contain nonzero classes of odd degree.

Moreover, $X$ is a toric Fano manifold, whose associated fan  in $\mathbb{R}^{n+1}$ has exactly $(n+3)$ primitive ray generators $b_i$, given by   (see e.g.  \cite{CLS}  for more details on toric geometry)
\begin{equation}\label{torbb}
    \begin{aligned}
          b_i&=e_i:=(0, \ldots, 0, 1, 0, \ldots, 0),\quad \mbox{for } 1\leq i\leq n+1,\\
    b_{n+2}&=-\sum_{i=1}^{n}e_i+3e_{n+1}, \qquad   b_{n+3}=-e_{n+1}.
    \end{aligned}
\end{equation}
Consequently, the Hori-Vafa superpotential $f$ mirror to   $X$ can be immediately read off, which is a holomorphic function $f: (\bC^\times)^{n+1}\rightarrow \bC$  defined by
\begin{align}\label{Laurent}
  f(\mathbf{x})=\sum_{i=1}^{n+3}\mathbf{x}^{b_i}=x_1+x_2+\cdots+x_{n+1}+\frac{x_{n+1}^3}{x_1x_2\cdots x_{n}}+\frac{1}{x_{n+1}}.
\end{align}

One remarkable statement in mirror symmetry for $X$ gives a ring isomorphism between $QH^*(X)$ and the Jacobi ring $\mbox{Jac}(f)$. Here $QH^*(X)=(H^*(X)\otimes_{\mathbb{C}} \mathbb{C}[q_1, q_2], \star)$ denotes the (small) quantum cohomology ring of $X$. It is a deformation of the classical cohomology ring $H^*(X)$, by incorporating   genus-zero, three-pointed Gromov-Witten invariants of $X$ into the quantum product (see e.g. \cite{CoKa} for more details). Precisely, there is an isomorphism $\Psi$ of $\mathbb{C}$-algebras \cite{Baty},   {\upshape  $$\Psi: QH^*(X)|_{(q_1, q_2)=(1,1)}\longrightarrow \mbox{Jac}(f)=\bC[x_1^{\pm 1},...,x_{n+1}^{\pm 1}]/(x_1\partial_{x_1}f,...,x_{n+1}\partial_{x_{n+1}}f).$$
     }
      Here we specify $\mathbf{q}=\mathbf{1}$ for only introducing  $f$ without deformation. Moreover, we have
\begin{prop}[\protect{\cite[Corollary G]{OsTy}}]\label{eigncrit}
   $\Psi(c_1(X))=[f]$. Namely, eigenvalues of $\hat c_1$ on $QH^*(X)|_{\mathbf{q}=\mathbf{1}}$, with multiplicities counted, coincide with the critical values of  $f$.
\end{prop}
\subsection{Distribution of critical values of $f$}

Assume  $3$ and $n+1$   to be coprime, i.e., $3\nmid n+1$.

Critical points $\mathbf{x}=(x_1, \ldots, x_{n+1})$ of $f$ are the solutions to the system of equations  $\partial_{x_i}f=0, \, 1\le i\le n+1$.
A simple calculation shows that
\begin{equation*}
    x_1=x_2=\cdots=x_n=:x \in \mathbb{C}^\times.
\end{equation*}
Denote  $y:=x_{n+1}\in\bC^\times$.  Then the system of equations $\frac{\partial f}{\partial x_i}=0$ can be reduced to
\begin{equation}\label{eqnnn}
        1-\dfrac{y^3}{x^{n+1}}=0,\quad
        1+\dfrac{3y^{2}}{x^n}-\dfrac{1}{y^2}=0,\quad (x,y)\in(\bC^\times)^2.
\end{equation}

Since $3$ and $n+1$ are coprime,  $l_1\cdot 3 + l_2 \cdot(n+1)=1$ for some integers $l_1, l_2$.  By setting $t=y^{l_2}x^{l_1}$,
we obtain the one-to-one parameterization $(x, y) =(t^{3}, t^{n+1})$ with $t\in\mathbb{C}^\times$ for solutions to the first equation in \eqref{eqnnn}. Therefore the system
\eqref{eqnnn} is equivalent to $h(t)=1$ with
\begin{equation}\label{contr}
    h(t):=t^{2n+2}+3t^{n+4}.
\end{equation}
Consequently, the  critical values of $f$ (with multiplicity counted) are precisely given by  $g(\alpha)$ at the roots $\alpha$ of $h-1$, and we further notice $g(\alpha)=\tilde g(\alpha)=\check g(\alpha)$, where
\begin{equation*}\label{defhg}
g(t):=n t^3+t^{n+1}+t^3+\frac{1}{t^{n+1}}, \quad \Tilde{g}(t):=(n-2)t^3+\frac{2}{t^{n+1}}, \quad \check g (t):=2 t^{n+1}+(n+4)t^3.
\end{equation*}
 Therefore we are led to   the following optimization problem in nonlinear programming.
\begin{prob}\label{NLP1}
  $\textnormal{Maximize }\,\,   |g(t)| \quad
\textnormal{subject to }\,\, h(t)=1,\,\, t\in \mathbb{C}^\times.$
  \end{prob}

Viewed as a real function, $h(t)$ is strictly increasing on $\mathbb{R}_{\geq0}$. Note $h(0)<1<h(1)$. Hence $h(t)-1$ has a unique positive single root $a_+$ with
\[a_+\in (0,1).\]
Moreover, the following lemma follows immediately from the observation \[\lim_{n\rightarrow+\infty}h(1-\frac{1}{n})=\lim_{n\rightarrow+\infty}\big((1-\frac{1}{n})^{2n+2}+3(1-\frac{1}{n})^{n+4}\big)=\frac{1}{e^2}+\frac{3}{e}>1.\]
\begin{lemma}\label{rangea}
    There exists an integer $N_1>0$, such that $a_+<1-\frac{1}{n}$ for any $n\ge N_1$.
\end{lemma}

 Recall   Rouch\'e's theorem in classical complex analysis   (see, e.g., \cite[Theorem 6.24]{GKR}).
\begin{lemma}[Rouch\'e's theorem]
    Let $f_1$ and $f_2$ be holomorphic functions  on $\{z\in \mathbb{C}\mid |z|\leq R\}$ where $R>0$. If $|f_2|<|f_1|$ on $\{z\in \mathbb{C}\mid |z|=R\}$, then $f_1$ and $f_1+f_2$ have the same number of zeros (counted with multiplicity) in $\{z\in \mathbb{C}\mid |z|\leq R\}$.
\end{lemma}

We say that the subset $\{z\in\mathbb{C}\mid |z|\leq 1\}$ is the unit disc.
\begin{prop}\label{mainprop1} There are exactly $n+4$ roots $\alpha$ of $h-1$ in the unit disc; in particular, they are all in  $\{z\in \mathbb{C}\mid a_+\leq |z|\leq 1\}$. Moreover, there exists an integer $N_2>N_1$, such that for any $n\geq N_2$,    $$|g(\alpha)|<n+2.$$ 
\end{prop}
\begin{proof}
Applying Rouch\'e's theorem to $3t^{n+4}$ and $t^{2n+2}-1$  on the unit circle, we have $$|3t^{n+4}|=3>2=|t|^{2n+2}+1\geq |t^{2n+2}-1|.$$
 Thus $h(t)-1=3t^{n+4}+t^{2n+2}-1$ has exactly   $n+4$ roots $\alpha$ in the unit disk.
          By the maximum modulus principle, for any $|z|<a_+$,
      $|h(z)|<\max\limits_{|\beta|=a_+} |h(\beta)|\leq h(a_+)=1$.
 That is, $h(t)-1$ has no roots in $\{z\mid |z|<a_+\}$. Hence, the first statement follows.

 By the maximum modulus principle for $\Tilde{g}$, when $n\geq2$, we have
 \begin{equation*}
 |g(\alpha)|=|\Tilde{g}(\alpha)|\leq \max_{|t|\in \{a_+, 1\}}|\Tilde{g}(t)|\max \{\Tilde{g}(a_+),\Tilde{g}(1)\}=\max\{\Tilde{g}(a_+), n\}.
\end{equation*}
Noting that   $\check{g}$ is strictly increasing   in $\mathbb{R}_{\ge0}$, we have  $\Tilde{g}(a_+)=\check{g}(a_+)<\check{g}(1-\frac{1}{n})$ for $n\geq N_1$, by Lemma \ref{rangea}.  Note $\lim\limits_{n\to \infty} 2(1-{1\over n})^{(n+1)}={2\over e}$ and $
\lim\limits_{n\to \infty} {(-4 + 11 n - 9 n^2)\over n^3}=0$. Take  $\epsilon:={1\over 2}-{1\over e}>0$. Then there exists $N_2>N_1$ such that for any $n\geq N_2$,
\begin{align*}
  \check g(1-{1\over n})  &
=2(1-\frac{1}{n})^{n+1}+{(-4 + 11 n - 9 n^2)\over n^3}+(n+1)\\
&< ({2\over e}+\epsilon)+ \epsilon+(n+1)=n+2.
\end{align*}
Hence, $|g(\alpha)|\leq \max\{\Tilde{g}(a_+), n\}\leq \max\{\check g(1-{1\over n}), n\}<n+2$, whenever $n\geq N_2$.
\end{proof}


It remains to discuss the  $n-2$ roots of   $h-1$ outside the   unit disk, which  require more   careful analysis. The following Lemma \ref{explemma} will be used in the proof of Proposition \ref{mainprop2}.

\begin{lemma}\label{explemma}
For any $0<r<\frac{1}{2}$ and $n>10$, we have $$\max_{|z|=r}|(3-z)^3({n-2\over n}+{2z\over n})^{n-2}|=(3-r)^3 ({n-2\over n}+{2r\over n})^{n-2}.$$
\end{lemma}
\begin{proof}
   For $|z|=r$, we note $s={\rm Re}(z)\in [-r, r]$  and write
   \begin{equation}\label{absolutevalue}
         |(3-z)^3({n-2\over n}+{2z\over n})^{n-2}|^2={\theta(s)\over n^{2n-4}},
    \end{equation}
  where $$\theta(s)=(9-6s+r^2)^3((n-2)^2+(4n-8)s+4r^2)^{n-2}.$$
By direct calculation, we have
\begin{align*}
    {d\theta\over ds}(s)&=2 (9 + r^2 - 6 s)^2 \big((-2 + n)^2 + 4 r^2 + 4 (-2 + n) s\big)^{-3 +  n} \cdot\\
    &\qquad \big(2 (-2 + n)^2 (9 + r^2 - 6 s) -
   9 ((-2 + n)^2 + 4 r^2 + 4 (-2 + n) s)\big)\\
   &> 2\cdot 6^2\cdot ((n-2)^2-2(n-2))^{-3+n}\cdot (2(n-2)^2\cdot 6-9((n-2)^2+1+2(n-2))>0.
\end{align*}
    Thus \eqref{absolutevalue} is maximized when $s=\mathrm{Re}(z)=r$, i.e., $z=r$.
\end{proof}

Let   $r_1:=0.96$, $r_2:=1.2$, $\varepsilon_0:=0.001$,  and $i\in \{1, 2\}$. One can check that
\begin{equation}\label{varep}
    \varepsilon_0<\min\{3e^{r_1}-e^{2r_1}-1, e^{2r_2}-3e^{r_2}-1\}
    \quad\mbox{and}\quad (3-{1\over e^{r_i}})^3e^{{2\over e^{r_i}}+\varepsilon_0}<e^4-2\varepsilon_0.
\end{equation}
These inequalities will be used in the proof of Proposition \ref{mainprop2}.

\begin{prop}\label{mainprop2}
There are exactly $n-2$ roots $\alpha$ of $h-1$ in   $\{z\in \mathbb{C}\mid  |z|> 1\}$. Moreover, there exists an integer $N_3>10$, such that for any $n\geq N_3$, $$1+{r_1\over  n}\leq |\alpha|\leq 1+{r_2\over n}\quad\mbox{and}\quad |g(\alpha)|<n+2.$$
\end{prop}
\begin{proof} Notice the following two limits
\begin{align*}
    \lim_{n\rightarrow+\infty}\big( 3(1+\frac{r_1}{n})^{n+4}-(1+\frac{r_1}{n})^{2n+2}-1\big)&=3e^{r_1}-e^{2r_1}-1 >\varepsilon_0,\\
\lim_{n\rightarrow+\infty}\big( (1+\frac{r_2}{n})^{2n+2}-3(1+\frac{r_2}{n})^{n+4}-1\big)&=e^{2r_2}-3e^{r_2}-1 >\varepsilon_0,
\end{align*}
where the inequality follows from \eqref{varep}. Thus there exists $N'_3>2$, such that for any $n\geq N'_3$, the following (i) and (ii) hold.
\begin{enumerate}[label=(\roman*)]
    \item For any $|t|=1+\frac{r_1}{n}$,
     $$|3t^{n+4}|-|t^{2n+2}-1|\geq 3(1+\frac{r_1}{n})^{n+4}-(1+\frac{r_1}{n})^{2n+2}-1>\varepsilon_0>0.$$
    \item For any $|t|=1+\frac{r_2}{n}$,
     $$|t^{2n+2}|-|3t^{n+4}-1|\geq (1+\frac{r_2}{n})^{2n+2}-3(1+\frac{r_2}{n})^{n+4}-1>\varepsilon_0>0.$$
\end{enumerate}
By applying Rouch\'e's theorem to (ii),   all the  $2n+2$ roots $\beta$ of $h-1$ satisfy
 $|\beta|\leq 1+{r_2\over n}$.
By applying Rouch\'e's theorem to (i), $h-1$ has exactly $n+4$ roots $\gamma$ with $|\gamma|\leq 1+{r_1\over n}$. Then we conclude  that  $1+\frac{r_1}{n}\le|\alpha|\le 1+\frac{r_2}{n}$,  since $h-1$ already has   $n+4$ roots  inside the unit disk by Proposition \ref{mainprop1}.

Note  $\alpha^{n-2}=\frac{1}{\alpha^{n+4}}-3$. We have the following (in)equalities.

\begin{align}
|{\Tilde{g}(\alpha)\over n}|^{n-2} &=\big|(\frac{1}{\alpha^{n+4}}-3)^3(1-\frac{2-{2\over\alpha^{n+4}}}{n})^{n-2}\big| \nonumber\\
 \label{ine1}   &\leq \max_{|t|\in \{1+{r_1\over n}, 1+{r_2\over n}\}} \big|(\frac{1}{t^{n+4}}-3)^3(1-\frac{2-{2\over t^{n+4}}}{n})^{n-2}\big|  \\
\label{ine2} &= \max_{|t|\in \{1+{r_1\over n}, 1+{r_2\over n}\}} \big|(\frac{1}{|t|^{n+4}}-3)^3(1-\frac{2-{2\over |t|^{n+4}}}{n})^{n-2}\big|\quad(\mbox{for } n\geq N_4)\\
 \label{ine3}&<\max_{i\in \{1, 2\}}\big|(\frac{1}{(1+{r_i\over n})^{n+4}}-3)^3(1-\frac{2-{2\over {  e^{r_i}}}-\varepsilon_0}{n})^{n-2}\big|
 \quad\,\,\,(\mbox{for } n\geq N_4)\\
 \label{ine4} &<\max_{i\in \{1, 2\}}(3-\frac{1}{e^{r_i}})^3e^{{2\over {  e^{r_i}}}+\varepsilon_0-2}+\varepsilon_0
 \qquad\qquad\qquad\qquad\quad\,\,\, (\mbox{for } n\geq N_5)\\
  \label{ine5}&<e^2-\varepsilon_0
  \\
  \label{ine6}&<(1+\frac{2}{n})^{n-2}
 \qquad \qquad\qquad\qquad\qquad\quad  \qquad\qquad\quad\quad\,\,(\mbox{for } n\geq N_6).
\end{align}
Here the inequality \eqref{ine1} follows from the maximum modulus principle.
Since $$\lim\limits_{n\to \infty}{1\over (1+{r_i\over n})^{n+4}}={1\over e^{r_i}}<{1\over e^{r_i}}+{\varepsilon_0\over 2}<{1\over 2},$$ there exists $N_4>10$ such that  $\frac{1}{(1+{r_i\over n})^{n+4}}<{1\over e^{r_i}}+{\varepsilon_0\over 2}$ holds for any $n\geq N_4$ and   $i\in\{1, 2\}$. Consequently, the equality \eqref{ine2} holds by Lemma \ref{explemma}. The inequalities \eqref{ine3}, \eqref{ine4}, \eqref{ine6} follow directly from the definition of limit, and the inequality \eqref{ine5} holds by \eqref{varep}.

 Hence, we are done, by taking $N_3=\max\{N'_3,N_4, N_5, N_6\}$ and noting $g(\alpha)=\Tilde{g}(\alpha)$.
 \end{proof}

\begin{thm}\label{mainthm2}
     Let $X= \mathbb{P}_{\mathbb{P}^{n}}(\mathcal{O}\oplus\mathcal{O}(3))$, where   $n>\max\{N_2,N_3\}$ and $3\nmid n+1$. Then    \[\rho<\dim X+1=n+2.\]
\end{thm}
\begin{proof}
  Notice $\rho=\max\limits_{h(t)=1} |g(t)|$,
   by Proposition \ref{eigncrit} and the analysis at the beginning of this subsection.  Therefore the statement is a direct consequence of the combination of
   Propositions \ref{mainprop1} and \ref{mainprop2}.
\end{proof}
\begin{remark}
  When   $3\mid n+1$, the curve $x^{n+1}-y^3=0$ is reducible, so  the parameterzation $(x, y)=(t^3, t^{n+1})$ is not sufficient. A more involved analysis similar to that for $\mathbb{P}_{\mathbb{P}^n}(\mathcal{O}\oplus \mathcal{O}(n-1))$ in \cite{GHIKLS} may be needed, in order to remove the assumption $3\nmid n+1$.
\end{remark}

\begin{remark}
    Numerical computations by Mathematica 10.0 show that   $n=16$ is the smallest number such that  Galkin's lower bound conjecture does not hold for $\mathbb{P}_{\mathbb{P}^{n}}(\mathcal{O}\oplus \mathcal{O}(3))$, while
     this conjecture does   hold for  $\mathbb{P}_{\mathbb{P}^{n}}(\mathcal{O}\oplus \mathcal{O}(a))$  with $a\in \{2, 4, 5\}$ and $n< 30$.

    In addition to the examples in \cite{GHIKLS},   $\mathbb{P}_{\mathbb{P}^{16}}(\mathcal{O}\oplus \mathcal{O}(3))$      gives a negative answer to Conjecture $\mathcal{O}$ as well.
  \end{remark}

\end{document}